\documentclass[12pt,twoside,reqno]{amsart}
\usepackage[latin1]{inputenc}
\usepackage[italian,english]{babel}
\usepackage{amssymb,latexsym}
\usepackage{amsmath,amsfonts,amsthm}
\usepackage{paralist}
\usepackage{color}

\numberwithin{equation}{section}

\newcommand{\R}{\mathbb{R}}
\newcommand{\N}{\mathbb{N}}

\newcommand{\J}{\mathcal{J}}
\newcommand{\I}{\mathcal{I}}
\newcommand{\s}{\sharp}

\DeclareMathOperator{\dive}{div}

\newtheorem{lem}{Lemma}
\newtheorem{thm}{Theorem}

\newtheorem{defn}{Definition} 
\theoremstyle{remark}
\newtheorem{remark}{Remark}

\begin{document}
\title{Ground states solutions for a non-linear equation involving a pseudo-relativistic Schr\"odinger operator}
\author[Vincenzo Ambrosio]{Vincenzo Ambrosio}

\address{%
Dipartimento di Matematica e Applicazioni\\ Universit\`a degli Studi "Federico II" di Napoli\\
via Cinthia, 80126 Napoli, Italy}

\email{vincenzo.ambrosio2@unina.it}



\date{}

\begin{abstract}
In this paper we prove the existence, regularity and symmetry of a ground state for a nonlinear equation in the whole space, involving a pseudo-relativistic Schr\"odinger operator.
\end{abstract}

\maketitle

\section{Introduction}

Recently, the study of fractional and nonlocal operators of elliptic type has attracted the attentions of many mathematicians. These operators arises in many different areas of research such as optimization, finance, minimal surfaces, phase transitions, quasi-geostrophic flows, crystal dislocation, anomalous diffusion, conservation laws and ultra-relativistic limits of quantum mechanics. For more details and applications we refer to \cite{bertoin, BKW, CSM, CSS, CV, CT, DPV, FJLL, FL1, FL2, FQT, silvestre, SV, toland} and references therein.

\noindent
In this paper we consider the following nonlinear fractional equation 
\begin{equation}\label{P}
[(-\Delta+m^{2})^{s}-m^{2s}]u+\mu u=|u|^{p-2}u \mbox{ in } \R^{N}
\end{equation}
where $s\in (0,1)$, $N>2s$, $p\in (2, \frac{2N}{N-2s})$, $m\geq 0$ and $\mu>0$.


The fractional operator 
\begin{equation}\label{operator}
(-\Delta+m^{2})^{s}
\end{equation}
which appears in (\ref{P}) is defined in the Fourier space by setting 
\begin{equation*}
\mathcal{F}(-\Delta+m^{2})^{s}u(k)=(|k|^{2}+m^{2})^{s} \mathcal{F}u(k)
\end{equation*}
or via Bessel potential, that is for every $u\in C^{\infty}_{c}(\R^{N})$
$$
(-\Delta+m^{2})^{s}u(x)=c_{N,s} m^{\frac{N+2s}{2}} P.V. \int_{\R^{N}} \frac{u(x)-u(y)}{|x-y|^{\frac{N+2s}{2}}} K_{\frac{N+2s}{2}} (m|x-y|) dy + m^{2s} u(x)
$$
for every $x\in \R^{N}$: see \cite{FF}. Here P.V. stands for the Cauchy principal value, 
$$
c_{N,s} = 2^{-\frac{N+2s}{2} +1} \pi^{-\frac{N}{2}} 2^{2s} \frac{s(1-s)}{\Gamma(2-s)}
$$
and $K_{\nu}$ denotes the modified Bessel function of the second kind with order $\nu$. 

When $s=1/2$ the operator (\ref{operator}) has a clear meaning in quantum mechanic: it corresponds to the free Hamiltonian of a free relativistic particle of mass $m$.
We remind that the study of $\sqrt{-\Delta+m^{2}}$ has been strongly influenced by papers  \cite{LY1, LY2}  of Lieb and Yau on the stability of relativistic matter. For  more recent works about this topic one can see \cite{A1, ChS, CN, FL, FJLL, FL2}.\\
Let us point out that the operator (\ref{operator}) has a deep connection with the Stochastic Process theory: in fact $(-\Delta+m^{2})^{s}-m^{2s}$ is an infinitesimal generator of a Levy process called the $2s$-stable relativistic process: see \cite{CMS, ryznar}.

In the present paper we are interested in the study of the ground states of (\ref{P}).
Such problems are motivated in particular by the search for certain kinds of solitary waves (stationary states) in nonlinear equations of the Klein-Gordon or Schr\"odinger
type.\\
In the celebrated papers \cite{BL1, BL2} Berestycki and Lions studied the existence of ground state of the following problem
\begin{equation}\label{BLP}
\left\{
\begin{array}{ll}
-\Delta u =g(u)&\mbox{ in }  \R^{N} \\
u\in H^{1}(\R^{N})  &\mbox{ and } u\not\equiv 0
\end{array}
\right.
\end{equation}
where $N\geq 3$, $g$ is a real continuous function such that 
\begin{enumerate}
\item $\lim_{t\rightarrow 0} \frac{g(t)}{t}=-\mu<0$;
\item $\lim_{t\rightarrow +\infty} \frac{g(t)}{t^{\frac{N+2}{N-2}}} \leq 0$;
\item there exists $t_{0}>0$ such that $G(t_{0})=\int_{0}^{t_{0}} g(z) dz>0$.
\end{enumerate}
In \cite{BL1}, solutions are constructed by solving the minimization problem
$$
\inf\Bigl \{ \int_{\R^{N}}|\nabla u|^{2} dx : \int_{\R^{N}} G(u) dx=1 \Bigr\}.
$$
Then a rescaling from $u$ to a suitable $ u(t^{-1}x)$ produces a solution of (\ref{BLP}) by absorbing the Lagrange multiplier. 

More recently, in \cite{DPV} authors studied, in a similar way, a non-local version of the above problem with a particular nonlinearity:  
\begin{equation}\label{DPV}
\left\{
\begin{array}{ll}
(-\Delta)^{s}u =-\mu u+|u|^{p-2}u &\mbox{ in }  \R^{N} \\
u\in H^{s}(\R^{N})  &\mbox{ and } u\not\equiv 0
\end{array}.
\right.
\end{equation}
Our aim is to prove that their result holds again when we replace $(-\Delta)^{s}$ by $(-\Delta+m^{2})^{s}-m^{2s}$. 
Unfortunately the method used in \cite{DPV} does not work in the case of the Bessel operator (\ref{operator}).
The most important difference between $(-\Delta)^{s}$ and (\ref{operator}) is that the first has some scaling properties that the latter does not have. 
In fact, it is not hard to check that one can solve the minimum problem
$$
\inf\Bigl \{ \int_{\R^{N}} |(-\Delta+m^{2})^{s/2}u|^{2} dx : \int_{\R^{N}} G(u) dx=1 \Bigr\}
$$
in the space of radially symmetric functions. But, in this case, we are not able to absorb the  Lagrange multiplier because of the lack of scaling invariance for the Bessel fractional operator. 

It is worth remembering that, by using different variational techniques and a Pohozaev type identity, Chang and Wang \cite{CW} proved the existence of a ground state to the problem (\ref{DPV}) with a more general nonlinearity. 
In a future work we will try to extend their result to the more general operator $(-\Delta+m^{2})^{s}-m^{2s}$.
 
Our first result can be stated as follows
\begin{thm}
Let $m>0$ and $\mu>0$. Then there exists a positive ground state solution $u\in H_{m}^{s}(\R^{N})$ to (\ref{P}) which is radially symmetric.
\end{thm}

The main difficulty of studying of problem (\ref{P}) is the non-local character of the involved operator. To overcome this difficulty, we use a technique very common in the recent literature: the Caffarelli-Silvestre extension method \cite{CS}. Such method consists to write a given nonlocal problem in a local way via the Dirichlet-Neumann map: this allow us to apply known variational techniques to these kind of problems; see for instance \cite{A2, CDDS,CTan, StingaTorrea}.
More precisely, for any $u\in H^{s}_{m}(\R^{N})$ there exists a unique $v\in H^{1}_{m}(\R^{N+1}_{+},y^{1-2s})$ weakly solving
\begin{equation*}
\left\{
\begin{array}{ll}
-\dive({y^{a} \nabla v})+m^{2}y^{a} v=0 &\mbox{ in }  \R^{N+1}_{+} \\
v(x,0)=u(x)   &\mbox{ on } \partial \R^{N+1}_{+} 
\end{array}
\right.
\end{equation*}
and such that
$$
\frac{\partial v}{\partial \nu^{a}}:=-\lim_{y \rightarrow 0} y^{a} \frac{\partial v}{\partial y}(x,y)=(-\Delta+m^{2})^{s} u(x) \mbox{ in } H_{m}^{-s}(\R^{N})
$$
where $a=1-2s\in (-1,1)$.

We will exploit this fact, and we will prove the existence, regularity results and qualitative properties of solutions to 
\begin{equation}\label{fracvaldinoci}
\left\{
\begin{array}{ll}
-\dive({y^{a} \nabla v})+m^{2}y^{a} v=0 &\mbox{ in }  \R^{N+1}_{+} \\
\frac{\partial v}{\partial \nu^{a}}=\kappa_{s}[m^{2s}u-\mu u+|u|^{p-2}u]   &\mbox{ on } \partial \R^{N+1}_{+} 
\end{array}.
\right.
\end{equation}

When we assume $m$ sufficiently small we are able to rediscover the result obtained by  \cite{DPV} (and \cite{CW}).
In fact, we are able to pass to the limit in (\ref{P}) as $m\rightarrow 0$ and we find a nontrivial solution to
\begin{equation*}
(-\Delta)^{s}u+\mu u=|u|^{p-2}u \mbox{ in } \R^{N}.
\end{equation*} 
More precisely we obtain
\begin{thm}
There exists a nontrivial ground state solution $u\in H^{s}(\R^{N})$ to (\ref{DPV})  such that $u$ is radially symmetric.
\end{thm}

\section{preliminaries}
In this section we collect some notations e basic facts we will use later.
Let $s\in (0,1)$ and $m>0$. 

Let
$$
\R^{N+1}_{+}=\{(x,y)\in \R^{N+1}:y>0 \}.
$$
We use the notation $|(x,y)|=\sqrt{|x|^{2}+y^{2}}$ to denote the euclidean norm in $\R^{N+1}_{+}$.
Let $R>0$. We denote by 
\begin{align*}
&B^{+}_{R}=\{(x,y)\in \R^{N+1}_{+}: |(x,y)|<R\}, \\
&\Gamma_{R}^{0}=\{(x,0)\in \partial \R^{N+1}_{+}: |x|<R \}
\end{align*}
and 
$$
\Gamma_{R}^{+}=\{(x,y)\in \overline{\R^{N+1}_{+}}: |(x,y)|= R \}.
$$
Let $q\in [1,\infty]$.
We denote by $|\cdot|_{L^{q}(A)}$ the $L^{q}$-norm in $A\subset \R^{N}$ and $||\cdot||_{L^{q}(B)}$ the $L^{q}$-norm in $B\subset \R^{N+1}_{+}$, respectively.

We define $H^{s}_{m}(\R^{N})$ as the completion of $C^{\infty}_{c}(\R^{N})$ with respect to the norm
$$
|u|^{2}_{H^{s}_{m}(\R^{N})}:= \int_{\R^{N}} (|k|^{2}+m^{2})^{s} |\mathcal{F}u(k)|^{2} dk<\infty,
$$
where $\mathcal{F}u(k)$ is the Fourier transform of $u$. When $m=1$ we write $H^{s}(\R^{N})= H^{s}_{1}(\R^{N})$ and $|\cdot|_{H^{s}(\R^{N})}= |\cdot|_{H^{s}_{1}(\R^{N})} $. 

We denote by $H^{1}_{m}(\R^{N+1}_{+},y^{1-2s})$ the completion of $C^{\infty}_{c}(\overline{\R^{N+1}_{+}})$ with respect to the norm
$$
||v||_{H^{1}_{m}(\R^{N+1}_{+},y^{1-2s})}^{2}:=\iint_{\R^{N+1}_{+}} y^{1-2s} (|\nabla v|^{2}+m^{2} v^{2}) dx dy<\infty.
$$
As before, when $m=1$, we use the notation $H^{1}(\R^{N+1}_{+},y^{1-2s})=H^{1}_{1}(\R^{N+1}_{+},y^{1-2s})$ and $||\cdot ||_{H^{1}(\R^{N+1}_{+},y^{1-2s})}= ||\cdot ||_{H^{1}_{1}(\R^{N+1}_{+},y^{1-2s})}$. 
 
It is possible to prove (see \cite{FF}) that it holds the following result:
\begin{thm}\label{thmterry}
There exists a trace operator $Tr: H^{1}_{m}(\R^{N+1}_{+},y^{1-2s}) \rightarrow H^{s}_{m}(\R^{N})$ such that
\begin{enumerate}
\item $Tr(v)=v(x,0)$ for any $v\in C^{\infty}_{c}(\overline{\R^{N+1}_{+}})$ \\
\item $\kappa_{s}|Tr(v)|_{H^{s}_{m}(\R^{N})}^{2}\leq ||v||_{H^{1}_{m}(\R^{N+1}_{+},y^{1-2s})}^{2}$ for any $v\in H^{1}_{m}(\R^{N+1}_{+},y^{1-2s})$,
where $\kappa_{s}=2^{1-2s}\frac{\Gamma(1-s)}{\Gamma(s)}$.

\noindent
Equality holds in $(2)$ for some function $v\in H^{1}_{m}(\R^{N+1}_{+},y^{1-2s})$ if and only if $v$ is a weak solution to 
$$
-\dive({y^{1-2s} \nabla v})+m^{2}y^{1-2s} v=0 \mbox{ in  } \R^{N+1}_{+}.
$$
\end{enumerate}
\end{thm}

As a consequence we can deduce
\begin{thm}\label{ext}
Let $u\in H_{m}^{s}(\R^{N})$. Then there exists a unique $v \in H^{1}_{m}(\R^{N+1}_{+}, y^{1-2s})$
 which solves 
\begin{equation}\label{extpb}
\left\{
\begin{array}{ll}
-\dive({y^{1-2s} \nabla v})+m^{2}y^{1-2s} v=0 &\mbox{ in }  \R^{N+1}_{+} \\
v=u  &\mbox{ on } \partial \R^{N+1}_{+} 
\end{array}.
\right.
\end{equation}
In particular 
$$
-\lim_{y \rightarrow 0} y^{1-2s} \frac{\partial v}{\partial y}(x,y)=\kappa_{s} (-\Delta+m^{2})^{s}u(x) \mbox{ in } H_{m}^{-s}(\R^{N}).
$$
\end{thm}

Since $H^{s}_{m}(\R^{N})\subset L^{q}(\R^{N})$ for any $q\leq \frac{2N}{N-2s}$, we can prove that 
\begin{thm}
For any $v \in H_{m}^{1}(\R^{N+1}_{+}, y^{1-2s})$ and for any $q\in \Bigl[2, \frac{2N}{N-2s} \Bigr]$
\begin{align}\label{SI}
C_{q,s,N}|u|_{L^{q}(\R^{N})}^{2} &\leq  
\kappa_{s}\int_{\R^{N}} (|k|^{2}+m^{2})^{s} |\mathcal{F}u(k)|^{2} dk \nonumber \\
&\leq \iint_{\R^{N+1}_{+}} y^{1-2s} (|\nabla v|^{2}+m^{2} v^{2}) dx dy
\end{align}
where $u(x)=v(x,0)$ is the trace of $v$ on $\partial \R^{N+1}_{+}$.
\end{thm}

In the sequel we will exploit Theorem \ref{ext} and we will look for the solutions to the following problem
\begin{equation}
\left\{
\begin{array}{ll}
-\dive({y^{a} \nabla v})+m^{2}y^{a} v=0 &\mbox{ in }  \R^{N+1}_{+} \\
\frac{\partial v}{\partial \nu^{a}}=\kappa_{s}[m^{2s}u-\mu u+|u|^{p-2}u]   &\mbox{ on } \partial \R^{N+1}_{+} 
\end{array}
\right.
\end{equation}
where 
$$
\frac{\partial v}{\partial \nu^{a}}:=-\lim_{y \rightarrow 0} y^{a} \frac{\partial v}{\partial y}(x,y)
$$
and $a=1-2s\in (-1,1)$.

For simplicity we will assume that $\kappa_{s}=1$.
Finally we recall the following compact embedding (see \cite{Lions}): 
\begin{thm}\label{compactradial}
Let 
$$
X^{m}_{rad}= \{v\in H_{m}^{1}(\R^{N+1}_{+} , y^{1-2s}) : v \mbox{ is radially symmetric with respect to } x\}.
$$ 
Then $X^{m}_{rad}$ is compactly embedded in $L^{q}(\R^{N})$ for any $q\in \Bigl(2, \frac{2N}{N-2s} \Bigr)$. 
\end{thm}

\section{Some results on elliptic problems involving $(-\Delta+m^{2})^{s}$:
Schauder estimates and maximum principles}

In this section we give some results about local Schauder estimates and maximum principle for problems involving the operator 
$$
-\dive(y^{1-2s} \nabla v)+m^{2}y^{1-2s}v.
$$
Firstly we give the following definition:
\begin{defn}
Let $R>0$ and $h\in L^{1}(\Gamma^{0}_{R})$. We say that $v\in H^{1}_{m}(B^{+}_{R})$ is a weak solution to 
\begin{equation*}
\left\{
\begin{array}{ll}
-\dive(y^{1-2s} \nabla v)+m^{2}y^{1-2s}v=y^{1-2s}d  &\mbox{ in } B^{+}_{R} \\
\frac{\partial v}{\partial \nu^{1-2s}}=h &\mbox{ on } \Gamma^{0}_{R} 
\end{array}
\right.
\end{equation*}
if
\begin{equation*}
\iint_{B^{+}_{R}} y^{1-2s} [\nabla v \nabla \phi+m^{2} v \varphi] dx dy=\int_{\Gamma^{0}_{R} } h \varphi dx
\end{equation*}
for any $\varphi \in C^{1}(\overline{B_{R}^{+}})$ such that $\varphi=0$ on $\Gamma_{R}^{+}$.
\end{defn}

Now, we state the following several regularity results whose proof can be found in \cite{FF}.

\begin{thm}\label{thm7}
Let $a,b\in L^{q}(\Gamma_{R}^{0})$ for some $q>\frac{N}{2s}$ and $c,d\in L^{r}(B^{+}_{R},y^{1-2s})$ for some $r>\frac{N+2-2s}{2}$. Let $v\in H^{1}_{m}(B^{+}_{R},y^{1-2s})$ be a weak solution to 
\begin{equation*}
\left\{
\begin{array}{ll}
-\dive(y^{1-2s} \nabla v)+m^{2}y^{1-2s}v=y^{1-2s}d  &\mbox{ in } B^{+}_{R} \\
\frac{\partial v}{\partial \nu^{1-2s}}=a(x)v+b(x) &\mbox{ on } \Gamma^{0}_{R} 
\end{array}.
\right.
\end{equation*}
Then $v\in C^{0,\alpha}(\overline{B^{+}_{R/2}})$.
\end{thm}

\begin{thm}\label{thm8}
Let $a,b\in C^{k}(\Gamma_{R}^{0})$ for some $q>\frac{N}{2s}$ and $\nabla_{x} c,\nabla_{x} d\in L^{\infty}(B^{+}_{R})$ for some $k\geq 1$.
Let $v\in H^{1}_{m}(B^{+}_{R},y^{1-2s})$ be a weak solution to 
\begin{equation*}
\left\{
\begin{array}{ll}
-\dive(y^{1-2s} \nabla v)+m^{2}y^{1-2s}v=y^{1-2s}d &\mbox{ in } B^{+}_{R} \\
\frac{\partial v}{\partial \nu^{1-2s}}=a(x)v+b(x) &\mbox{ on } \Gamma^{0}_{R} 
\end{array}.
\right.
\end{equation*}
Then $v\in C^{i,\alpha}(\overline{B^{+}_{R/2}})$ for $i=1,\dots,k$.
\end{thm}

\begin{thm}\label{thm9}
Let $g\in C^{0,\gamma}(\Gamma_{R}^{0})$ for some $\gamma \in [0,2-2s)$. 
Let $v\in H^{1}_{m}(B^{+}_{R},y^{1-2s})$ be a weak solution to 
\begin{equation*}
\left\{
\begin{array}{ll}
-\dive(y^{1-2s} \nabla v)+m^{2}y^{1-2s}v= 0 &\mbox{ in } B^{+}_{R} \\
\frac{\partial v}{\partial \nu^{1-2s}}=g &\mbox{ on } \Gamma^{0}_{R} 
\end{array}.
\right.
\end{equation*}
Then, for any $t_{0}>0$ sufficiently small,  $y^{1-2s}\partial_{y} v \in C^{0,\alpha}([0,t_{0})\times \Gamma_{R/8})$.
\end{thm}

In the spirit of the paper \cite{CS1} we can prove the following maximum principles:

\begin{thm}(weak maximum principle)
Let $v\in H^{1}_{m}(B^{+}_{R},y^{1-2s})$ be a weak solution to 
\begin{equation}
\left\{
\begin{array}{ll}
-\dive(y^{1-2s} \nabla v)+m^{2}y^{1-2s}v \geq 0 &\mbox{ in } B^{+}_{R} \\
\frac{\partial v}{\partial \nu^{1-2s}}\geq 0 &\mbox{ on } \Gamma^{0}_{R} \\
v \geq 0 &\mbox{ on } \Gamma^{+}_{R}
\end{array}.
\right.
\end{equation}
Then $v\geq 0$ in $B^{+}_{R}$.
\end{thm}
\begin{proof}
It is enough to multiply the weak formulation of above problem by $v^{-}$. 
\end{proof}
\begin{remark}
We can deduce also the strong maximum principle: either $v\equiv 0$ or $v>0$ in $B^{+}_{R} \cup \Gamma _{R}^{0}$. In fact, $v$ can't vanish at an interior point  by the classical strong maximum principle for strictly elliptic operators. Finally the fact that $v$ can't vanish at a point in $\Gamma _{R}^{0}$ follows by the Hopf principle that we will proved below.
\end{remark}

\begin{thm}\label{hopf}
Let $C_{R}=B_{R}(0)\times (0,1)$ and $v\in H^{1}_{m}(C_{R},y^{1-2s}) \cap C(\overline{C_{R}})$ be a weak solution to 
\begin{equation}
\left\{
\begin{array}{ll}
-\dive(y^{1-2s} \nabla v)+m^{2}y^{1-2s}v \leq 0  &\mbox{ in } C_{R} \\
v>0  &\mbox{ in } C_{R} \\
v(0,0)=0  
\end{array}.
\right.
\end{equation}
Then
$$
\limsup_{y\rightarrow 0^{+}}-y^{1-2s} \frac{v(0,y)}{y}<0.
$$
If $y^{1-2s}v_{y}\in C(\overline{C_{R}})$ then 
$$
\frac{\partial v}{\partial y^{1-2s}}(0,0)<0.
$$
\end{thm}
\begin{proof}
We consider the function
$$
w_{A}(x,y)=y^{-1+2s} (y+Ay^{2}) \varphi(x)
$$
where $A>0$ is a constant that will be chosen later and $\varphi(x)$ is the first eigenfunction of $-\Delta_{x}+m^{2}$ in $B_{R/2}(0)$ with zero boundary condition.
Then we can conclude proceeding as in the proof of Proposition $4.11$ in \cite{CS1}.
\end{proof}

\begin{thm}
Let $d\in C^{0,\alpha}(\Gamma^{0}_{R})$ and $v\in H^{1}_{m}(B^{+}_{R},y^{1-2s})$ be a weak solution to 
\begin{equation}
\left\{
\begin{array}{ll}
-\dive(y^{1-2s} \nabla v)+m^{2}y^{1-2s}v = 0 &\mbox{ in } B^{+}_{R} \\
\frac{\partial v}{\partial \nu^{1-2s}}+d(x)v=0 &\mbox{ on } \Gamma^{0}_{R} \\
v \geq 0 &\mbox{ on } B^{+}_{R}
\end{array}.
\right.
\end{equation}
Then $v> 0$ in $B^{+}_{R}\cup \Gamma^{0}_{R}$ unless $v\equiv 0$ in $B^{+}_{R}$.
\end{thm}
\begin{proof}
By using Theorem \ref{thm7} and Theorem \ref{thm9} we know that $v$ and $y^{1-2s}v_{y}$ are $C^{0,\alpha}$ up to the boundary. So the equation  $\frac{\partial v}{\partial \nu^{1-2s}}+d(x)v=0$ is satisfied pointwise on $ \Gamma^{0}_{R}$. 
If $u$ is not identically 0 in $B^{+}_{R}$ then $u>0$ n $B^{+}_{R}$  by the strong maximum principle for the operator $L(v)=-\dive(y^{1-2s} \nabla v)+m^{2}y^{1-2s}v$. If $u(x_{0},0)=0$ at some point $(x_{0},0)\in \Gamma^{0}_{R}$, then $\frac{\partial v}{\partial y^{1-2s}}(x_{0},0)<0$. This gives a contradiction.
\end{proof}

\section{existence of ground state}

In this section we prove the existence of a ground state to (\ref{P}). 

Let us consider the following functional
\begin{equation}\label{I}
\mathcal{I}_{m}(v) =\frac{1}{2} \iint_{\R^{N+1}_{+}} y^{1-2s} (|\nabla v|^{2}+m^{2} v^{2}) dx dy-\frac{m^{2s}}{2} \int_{\R^{N}} v^{2}+\frac{\mu}{2} \int_{\R^{N}} v^{2} dx-\frac{1}{p} \int_{\R^{N}} |v|^{p} dx
\end{equation}
defined for any $v \in H_{m}^{1}(\R^{N+1}_{+}, y^{1-2s})$.

Firstly we note that 
$$
\iint_{\R^{N+1}_{+}} y^{1-2s} (|\nabla v|^{2}+m^{2} v^{2}) dx dy+(\mu-m^{2s})  \int_{\R^{N}} v^{2} dx
$$
is equivalent to the standard norm in $H_{m}^{1}(\R^{N+1}_{+}, y^{1-2s})$
$$
||v||_{H_{m}^{1}(\R^{N+1}_{+}, y^{1-2s})}^{2}=\iint_{\R^{N+1}_{+}} y^{1-2s} (|\nabla v|^{2}+m^{2} v^{2}) dx dy.
$$
In fact, if $\mu \geq m^{2s}$ then we have
$$
\iint_{\R^{N+1}_{+}} y^{1-2s} (|\nabla v|^{2}+m^{2} v^{2}) dx dy+(\mu-m^{2s})  \int_{\R^{N}} v^{2} dx \geq ||v||_{H_{m}^{1}(\R^{N+1}_{+}, y^{1-2s})}^{2}
$$
and by using (\ref{SI}) we get
$$
\iint_{\R^{N+1}_{+}} y^{1-2s} (|\nabla v|^{2}+m^{2} v^{2}) dx dy+(\mu-m^{2s})  \int_{\R^{N}} v^{2} dx \leq \Bigl(1+\frac{\mu-m^{2s}}{m^{2s}}\Bigr) ||v||_{H_{m}^{1}(\R^{N+1}_{+}, y^{1-2s})}^{2}.
$$
Now, we suppose $\mu<m^{2s}$.

Then
$$
\iint_{\R^{N+1}_{+}} y^{1-2s} (|\nabla v|^{2}+m^{2} v^{2}) dx dy+(\mu-m^{2s})  \int_{\R^{N}} v^{2} dx \leq ||v||_{H_{m}^{1}(\R^{N+1}_{+}, y^{1-2s})}^{2}
$$
and by (\ref{SI})
$$
\iint_{\R^{N+1}_{+}} y^{1-2s} (|\nabla v|^{2}+m^{2} v^{2}) dx dy+(\mu-m^{2s})  \int_{\R^{N}} v^{2} dx \geq \Bigl(1+\frac{\mu-m^{2s}}{m^{2s}}\Bigr) ||v||_{H_{m}^{1}(\R^{N+1}_{+}, y^{1-2s})}^{2}.
$$
Thus
\begin{align}\label{equivalent}
C_{1}(m,s,\mu)||v||_{H_{m}^{1}(\R^{N+1}_{+}, y^{1-2s})}^{2}&\leq ||v||_{H_{m}^{1}(\R^{N+1}_{+}, y^{1-2s})}^{2}+(\mu-m^{2s})  |v|_{L^{2}(\R^{N})}^{2} \nonumber \\
&\leq C_{2}(m,s,\mu) ||v||_{H_{m}^{1}(\R^{N+1}_{+}, y^{1-2s})}^{2}
\end{align}
and we set 
$$
||v||_{e,m}^{2}:=\iint_{\R^{N+1}_{+}} y^{1-2s} (|\nabla v|^{2}+m^{2} v^{2}) dx dy+(\mu-m^{2s})  \int_{\R^{N}} v^{2} dx.
$$
Now, in order to prove Theorem 1, we minimize $\mathcal{I}_{m}$ on the following Nehari manifold
$$
\mathcal{N}_{m}=\{ v\in H_{m}^{1}(\R^{N+1}_{+}, y^{1-2s})\setminus \{0\}: \mathcal{J}_{m}(v)=0 \}.
$$
where $\mathcal{J}_{m}(v)=\mathcal{I}_{m}'(v)v$ that is
\begin{align*}
\J_{m}(v)&= \iint_{\R^{N+1}_{+}} y^{1-2s} (|\nabla v|^{2}+m^{2} v^{2}) dx dy-m^{2s} \int_{\R^{N}} v^{2}+\mu \int_{\R^{N}} v^{2} dx- \int_{\R^{N}} |v|^{p} dx\\
&=||v||_{e,m}^{2}-\int_{\R^{N}} |v|^{p} dx.
\end{align*}
Finally we define
$$
c_{m}=\min_{v\in \mathcal{N}_{m}} \I_{m}(v).
$$

\begin{proof}(proof of Theorem 1)
We divide the proof into several steps.

{\it Step 1} The set $\mathcal{N}_{m}$ is not empty. 

Fix $v\in H_{m}^{1}(\R^{N+1}_{+}, y^{1-2s})\setminus \{0\}$. Then
$$
h(t):=\I_{m}(tv)=\frac{t^{2}}{2}||v||^{2}_{e,m}-\frac{t^{p}}{p}|v|_{L^{p}(\R^{N})}^{p}
$$
achieves its maximum in some $\tau>0$. By differentiating $h$ with respect to $t$ we have $\I_{m}'(\tau v)v=0$ and $\tau v\in \mathcal{N}_{m}$. 

{\it Step 2} Selection of an adequate minimizing sequence. 

We prove that there exists  a radially symmetric function $v\in \mathcal{N}_{m}$ such that $\I_{m}(v)=c_{m}$.

Let $(v_{j})\subset \mathcal{N}_{m}$ be a minimizing sequence for $\I_{m}$ and $u_{j}$ its trace. Let $\tilde{u}_{j}$ be the symmetric-decreasing rearrangement of $u_{j}$.
It is known (see \cite{Sbessel}) that 
\begin{align}\label{decrearr}
\int_{\R^{N}} (|k|^{2} + m^{2})^{s} |\mathcal{F}\tilde{u}_{j} (k)|^{2} dk \leq \int_{\R^{N}} (|k|^{2} + m^{2})^{s} |\mathcal{F}u_{j} (k)|^{2} dk
\end{align}
and $|\tilde{u}_{j}|_{L^{q}(\R^{N})}=|u_{j}|_{L^{q}(\R^{N})}$ for any $q\geq 1$. Now, let $\tilde{v}_{j}$ be the unique solution to 
\begin{equation}
\left\{
\begin{array}{lll}
-\dive(y^{1-2s} \nabla \tilde{v}_{j})+m^{2} \tilde{v}_{j} =0 &\mbox{ in } \R^{N+1}_{+}\\
\tilde{v}_{j}(x,0)=\tilde{u}_{j}(x)  &\mbox{ on } \partial \R^{N+1}_{+}
\end{array}.
\right.
\end{equation}

We recall that 
\begin{equation}\label{vincebello}
||\tilde{v}_{j}||_{H_{m}^{1}(\R^{N+1}_{+}, y^{1-2s})}=|\tilde{u}_{j}|_{H^{s}_{m}(\R^{N})},
\end{equation}
by using $(2)$ in Theorem \ref{thmterry}.  

Taking into account (\ref{SI}), (\ref{decrearr}) and (\ref{vincebello}) we get 
\begin{align*}
\iint_{\R^{N+1}_{+}}  y^{1-2s} (|\nabla \tilde{v}_{j}|^{2} +m^{2} \tilde{v}_{j}^{2})\, dxdy &= \int_{\R^{N}} (|k|^{2} + m^{2})^{s} |\mathcal{F}\tilde{u}_{j} (k)|^{2} dk \\
&\leq \int_{\R^{N}} (|k|^{2} + m^{2})^{s} |\mathcal{F}u_{j} (k)|^{2} dk \\
&\leq \iint_{\R^{N+1}_{+}}  y^{1-2s} (|\nabla v_{j}|^{2} +m^{2} v_{j}^{2})\, dxdy, 
\end{align*}
so we deduce that
$$
\J_{m}(\tilde{v}_{j})  \leq \J_{m}(v_{j}) =0 \,
\mbox{ and } \, \I_{m}(\tilde{v}_{j}) \leq \I_{m}(v_{j}). 
$$
Proceeding as in the proof of Step $1$, we can find $t_{j}>0$ such that $\J_{m}(t_{j} \tilde{v}_{j})=0$. By using the fact that $\J_{m}(\tilde{v}_{j}) \leq 0$ we can see that $t_{j}\leq 1$. 

Since $\J_{m}(v)=\I_{m}'(v) \, v$ we have
\begin{align*}
0&=\J_{m}(t_{j} \tilde{v}_{j})\\
&= \Bigl [\I_{m}(t_{j} \tilde{v}_{j}) + \frac{1}{p} \int_{\R^{N}} |t \tilde{v}_{j}|^{p}  dx - \frac{1}{2} \int_{\R^{N}} |t_{j} \tilde{v}_{j}|^{p} dx\Bigr]
\end{align*}
and by using the fact that $0<t_{j}\leq 1$ we obtain 
\begin{align*}
\I_{m}(t_{j} \tilde{v}_{j}) &= \Bigl(\frac{1}{2}-\frac{1}{p} \Bigr) \int_{\R^{N}} |t_{j} \tilde{v}_{j}|^{p} dx \\
&\leq \Bigl(\frac{1}{2}-\frac{1}{p} \Bigr) \int_{\R^{N}} |\tilde{v}_{j}|^{p} dx \\ 
&= \Bigl(\frac{1}{2}-\frac{1}{p} \Bigr) \int_{\R^{N}} |v_{j}|^{p} dx \\ 
&=\I(v_{j}).
\end{align*}
Then $w_{j} :=t_{j} \tilde{v}_{j} $ is a minimizing sequence, radially symmetric with respect to $x$, of $\I_{m}$ on $\mathcal{N}_{m}$. By using (\ref{equivalent}) we get
\begin{align*}
\Bigl(\frac{1}{2}-\frac{1}{p}\Bigr) C_{1}(m,s,\mu)||w_{j}||_{H_{m}^{1}(\R^{N+1}_{+}, y^{1-2s})}^{2}\leq \Bigl(\frac{1}{2}-\frac{1}{p}\Bigr) ||w_{j}||^{2}_{e,m} =\I_{m}(w_{j}) <C
\end{align*}   
then, by using Theorem \ref{compactradial} we can assume that 
\begin{align}
&w_{j}\rightharpoonup w \mbox{ in }  H_{m}^{1}(\R^{N+1}_{+}, y^{1-2s}) \label{convergence1} \\
&w_{j} (\cdot, 0) \rightarrow w(\cdot, 0) \mbox{ in } L^{q}(\R^{N}) \quad \forall q\in \Bigl(2, \frac{2N}{N-2s} \Bigr) \label{convergence2}.  
\end{align}
Then we deduce that $\I_{m}(w)\leq c$ and $\J_{m}(w)\leq 0$. Now we claim that $w$ is not identically zero. We check this, we assume by contradiction that $w= 0$. By using the fact that $\J_{m}(w_{j})=0$ and (\ref{SI}), we can see that 
\begin{align*}
\int_{\R^{N}} |w_{j}|^{p} \, dx &= \iint_{\R^{N+1}_{+}} y^{1-2s} (|\nabla w_{j}|^{2}+m^{2} w_{j}^{2}) \, dx dy+(\mu-m^{2s})  \int_{\R^{N}} w_{j}^{2} \, dx \\
&\geq  C_{1}(m,s,\mu)||w_{j}||^{2}_{H_{m}^{1}(\R^{N+1}_{+}, y^{1-2s})} \\
&\geq C(m,s,\mu,N,p)  \Bigl(\int_{\R^{N}} |w_{j}|^{p} \, dx\Bigr)^{\frac{2}{p}}.
\end{align*} 
By using (\ref{convergence2}) we deduce that  $|w|_{L^{p}(\R^{N})}\geq C(m,s,\mu,N,p)^{1/p-2}>0$, which gets a contradiction.  Then, we can find $\tau \in (0, 1]$ such that $\I_{m}(\tau w) \leq c $ and $\J_{m}(\tau w) =0$. 

{\it Step 3} Conclusion.
Let $v$ be the minimizer obtained above. By using the fact that $v\in \mathcal{N}_{m}$ we have
\begin{align*}
\J_{m}'(v)v &= 2  \Bigl(\iint_{\R^{N+1}_{+}} y^{1-2s} (|\nabla v|^{2}+m^{2} v^{2}) \, dx dy+(\mu-m^{2s})  \int_{\R^{N}} v^{2} \, dx\Bigr) - p \int_{\R^{N}} |v|^{p} dx \\
&= (2-p) \int_{\R^{N}} |v|^{p} dx \neq 0. 
\end{align*}
As a consequence we can find a Lagrange multiplier $\lambda \in \R$ such that 
\begin{equation}\label{lagrange}
\I_{m}'(v) \varphi= \lambda \J_{m}'(v)\varphi
\end{equation}
for any $\varphi \in X_{rad}^{m}$. Taking $\varphi=v$ in (\ref{lagrange}) we deduce that $\lambda=0$ and $v$ is a nontrivial solution to (\ref{fracvaldinoci}).     
 
\end{proof}

\section{regularity and symmetry of solution to (\ref{P})}

\begin{lem}\label{lemmino}

Let $v\in H^{1}_{m}(\R^{N+1}_{+}, y^{1-2s})$ be a weak solution to 
\begin{equation}
\left\{
\begin{array}{ll}
-\dive(y^{1-2s} \nabla v)+m^{2}y^{1-2s}v =0 &\mbox{ in } \R^{N+1}_{+} \\
\frac{\partial v}{\partial \nu^{1-2s}}=m^{2s}v+f(x,v)  &\mbox{ on } \partial \R^{N+1}_{+}
\end{array},
\right.
\end{equation}
where $f(x,v)=-\mu v+|v|^{p-2}v$.
Then $v(\cdot,0)\in L^{q}(\R^{N})$ for all $q\in [2,\infty]$.
\end{lem}
\begin{proof}
We proceed as in the proof of Theorem 3.2. in \cite{CN}. 
Since $v$ is a weak solution to (\ref{P}), we know that
\begin{equation}\label{criticpoint}
\iint_{\R^{N+1}_{+}} y^{1-2s}(\nabla v \nabla \eta+m^{2}v \eta) \, dxdy=\int_{\R^{N}} m^{2s}v\eta+f(x,v)\eta \,dx
\end{equation}
for all $\eta\in H^{1}_{m}(\R^{N+1}_{+}, y^{1-2s})$.

Let $w=vv^{2\beta}_{K}\in H^{1}_{m}(\R^{N+1}_{+}, y^{1-2s})$ where $v_{K}=\min\{|v|,K\}$, $K>1$ and $\beta\geq 0$.
Taking $\eta=w$ in (\ref{criticpoint}) we deduce that 
\begin{align}\label{conto1}
\iint_{\R^{N+1}_{+}}  &y^{1-2s}v^{2\beta}_{K}(|\nabla v|^{2}+m^{2}v^{2}) \, dxdy+\iint_{D_{K,T}} 2\beta y^{1-2s}v^{2\beta}_{K} |\nabla v|^{2} \, dx dy  \nonumber \\
&=m^{2s}\int_{\R^{N}} v^{2} v^{2\beta}_{K} \,dx+ \int_{\R^{N}} f(x,v)vv^{2\beta}_{K} \,dx 
\end{align}
where $D_{K,T}=\{(x,y)\in \R^{N+1}_{+}: |v(x,y)|\leq K\}$. \\
It is easy to see that
\begin{align}\label{conto2}
\iint_{\R^{N+1}_{+}} &y^{1-2s}|\nabla (vv_{K}^{\beta})|^{2} dxdy \nonumber \\
&=\iint_{\R^{N+1}_{+}} y^{1-2s}v_{K}^{2\beta} |\nabla v|^{2} dxdy+\iint_{D_{K,T}} (2\beta+\beta^{2}) y^{1-2s}v_{K}^{2\beta} |\nabla v|^{2} dxdy.
\end{align}
Then, putting together (\ref{conto1}) and (\ref{conto2}) we get 
\begin{align}\label{S1}
&||vv_{K}^{\beta}||_{H^{1}_{m}(\R^{N+1}_{+}, y^{1-2s})}^{2}\\
&=\iint_{\R^{N+1}_{+}} y^{1-2s}[|\nabla (vv_{K}^{\beta})|^{2}+m^{2}v^{2}v_{K}^{2\beta}] dxdy \nonumber \\
&=\iint_{\R^{N+1}_{+}} y^{1-2s}v_{K}^{2\beta}[ |\nabla v|^{2}+m^{2}v^{2}] dxdy+\iint_{D_{K,T}} 2\beta \Bigl(1+\frac{\beta}{2}\Bigr) y^{1-2s}v_{K}^{2\beta} |\nabla v|^{2} dxdy \nonumber \\
&\leq c_{\beta} \Bigl[\iint_{\R^{N+1}_{+}} y^{1-2s}v_{K}^{2\beta}[ |\nabla v|^{2}+m^{2}v^{2}] dxdy+\iint_{D_{K,T}} 2\beta y^{1-2s}v_{K}^{2\beta} |\nabla v|^{2} dxdy\Bigr] \nonumber \\
&=c_{\beta} \int_{\R^{N}} m^{2s}v^{2}v_{K}^{2\beta} + f(x,v)v v_{K}^{2\beta} \,dx 
\end{align}
where $c_{\beta}=1+\frac{\beta}{2}$.
Then we deduce that
\begin{equation*}
m^{2s}v^{2}v_{K}^{2\beta} + f(x,v)v v_{K}^{2\beta}\leq (m^{2s}+C)v^{2}v_{K}^{2\beta}+C|v|^{p-2}v^{2}v_{K}^{2\beta} \, \mbox{ on } \R^{N}.
\end{equation*}
Now, we prove that
\begin{equation*}
|v|^{p-2}\leq 1+h \mbox{ on } \R^{N}
\end{equation*}
for some $h\in L^{N/2s}(\R^{N})$.
Firstly, we observe that
$$
|v|^{p-2}=\chi_{\{|v|\leq 1\}}|v|^{p-2}+\chi_{\{|v|>1\}}|v|^{p-2}\leq 1+\chi_{\{|v|>1\}}|v|^{p-2}\ \mbox{ on } \R^{N}.
$$
If $(p-2)\frac{N}{2s}<2$ then 
$$
\int_{\R^{N}} \chi_{\{|v|>1\}}|v|^{\frac{N}{2s}(p-2)} dx \leq \int_{\R^{N}} \chi_{\{|v|>1\}}|v|^{2} dx<\infty
$$
while if $2\leq (p-2)\frac{N}{2s}$ we have that $(p-2)\frac{N}{2s}\in [2,\frac{2N}{N-2s}]$.

Therefore, there exist a constant $c>0$ and a function $h\in L^{N/2s}(\R^{N})$, $h\geq 0$ and independent of $K$ and $\beta$, such that
\begin{equation}\label{S2}
m^{2s}v^{2}v_{K}^{2\beta} + f(x,v)vv_{K}^{2\beta}\leq (c+h)v^{2}v_{K}^{2\beta} \mbox{ on } \R^{N}.
\end{equation}
Taking into account (\ref{S1}) and (\ref{S2}) we have
\begin{equation*}
||vv_{K}^{\beta}||_{H^{1}_{m}(\R^{N+1}_{+},y^{1-2s})}^{2}\leq c_{\beta} \int_{\R^{N}} (c+h)v^{2}v_{K}^{2\beta} dx,
\end{equation*}
and by Monotone Convergence Theorem ($v_{K}$ is increasing with respect to $K$) we have as $K\rightarrow \infty$
\begin{equation}\label{i1}
|||v|^{\beta+1}||_{H^{1}_{m}(\R^{N+1}_{+},y^{1-2s})}^{2}\leq cc_{\beta} \int_{\R^{N}} |v|^{2(\beta +1)} dx + c_{\beta}\int_{\R^{N}}  h|v|^{2(\beta +1)}dx.
\end{equation}
Fix $M>0$ and let $A_{1}=\{h\leq M\}$ and $A_{2}=\{h>M\}$.

Then
\begin{equation}\label{i2}
\int_{\R^{N}}  h|v(\cdot,0)|^{2(\beta +1)} dx\leq M||v(\cdot,0)|^{\beta+1}|_{L^{2}(\R^{N})}^{2}+\varepsilon(M)||v(\cdot,0)|^{\beta+1}|_{L^{2^{\s}_{s}}(\R^{N})}^{2}
\end{equation}
where $\varepsilon(M)=\Bigl(\int_{A_{2}} h^{N/2s} dx \Bigr)^{\frac{2s}{N}}\rightarrow 0$ as $M\rightarrow \infty$.
Taking into account (\ref{i1}), (\ref{i2}), we get
\begin{equation}\label{regv}
|||v|^{\beta+1}||_{H^{1}_{m}(\R^{N+1}_{+},y^{1-2s})}^{2}\leq c_{\beta}(c+M)||v(\cdot,0)|^{\beta+1}|_{L^{2}(\R^{N})}^{2}+c_{\beta}\varepsilon(M)||v(\cdot,0)|^{\beta+1}|_{L^{2^{\s}_{s}}(\R^{N})}^{2}.
\end{equation}
By using (\ref{SI}) we know that
\begin{equation}\label{S3}
||v(\cdot,0)|^{\beta+1}|_{L^{2^{\s}_{s}}(\R^{N})}^{2}\leq C^{2}_{2^{\s}_{s}}|||v|^{\beta+1}||_{H^{1}_{m}(\R^{N+1}_{+},y^{1-2s})}^{2}.
\end{equation}
Then, choosing $M$ large so that $\varepsilon(M) c_{\beta} C^{2}_{2^{\s}}<\frac{1}{2}$,
and by using $(\ref{regv})$ and $(\ref{S3})$ we obtain
\begin{equation}\label{iter}
||v(\cdot,0)|^{\beta+1}|_{L^{2^{\s}_{s}}(\R^{N})}^{2}\leq 2 C^{2}_{2^{\s}_{s}} c_{\beta}(c+M)||v(\cdot,0)|^{\beta+1}|^{2}_{L^{2}(\R^{N})}.
\end{equation}
Then we can start a bootstrap argument: since $v(\cdot,0)\in L^{\frac{2N}{N-2s}}(\R^{N})$ we can apply (\ref{iter}) with $\beta_{1}+1=\frac{N}{N-2s}$ to deduce that $v(\cdot,0)\in L^{\frac{(\beta_{1}+1)2N}{N-2s}}(\R^{N})=L^{\frac{2N^{2}}{(N-2s)^{2}}}(\R^{N})$. Applying (\ref{iter}) again, after $k$ iterations, we find $v(\cdot,0)\in L^{\frac{2N^{k}}{(N-2s)^{k}}}(\R^{N})$, and so $v(\cdot,0)\in L^{q}(\R^{N})$ for all $q\in[2,\infty)$.

Now we prove that $v(\cdot,0)\in L^{\infty}(\R^{N})$.
Since $v(\cdot,0)\in L^{q}(\R^{N})$ for all $q\in[2,\infty)$ we have that $h\in L^{\frac{N}{s}}(\R^{N})$.

By using generalized H\"older inequality, we can see that
\begin{align}
\int_{\R^{N}} h|v|^{2(\beta+1)} dx &\leq |h|_{L^{\frac{N}{s}}(\R^{N})} ||v|^{\beta+1}|_{L^{2}(\R^{N})}  ||v|^{\beta+1}|_{L^{2^{\s}_{s}}(\R^{N})}  \nonumber \\
&\leq |h|_{L^{\frac{N}{s}}(\R^{N})} \Bigl(\lambda ||v|^{\beta+1}|_{L^{2}(\R^{N})}^{2}+\frac{1}{\lambda}  ||v|^{\beta+1}|_{L^{2^{\s}_{s}}(\R^{N})}^{2}\Bigr)
\end{align}
and by using (\ref{i1}) we deduce
\begin{align}\label{i3}
|||v|^{\beta+1}||_{H^{1}_{m}(\R^{N+1}_{+},y^{1-2s})}^{2}\leq  c_{\beta} (c+ |h|_{L^{\frac{N}{s}}(\R^{N})} \lambda) ||v|^{\beta+1}|_{L^{2}(\R^{N})}^{2}+\frac{c_{\beta}|h|_{L^{\frac{N}{s}}(\R^{N})}}{\lambda}  ||v|^{\beta+1}|_{L^{2^{\s}_{s}}(\R^{N})}^{2}.
\end{align}
Taking $\lambda$ such that
$$
\frac{c_{\beta}|h|_{L^{\frac{N}{s}}(\R^{N})}}{\lambda}C^{2}_{2^{s}_{s}}=\frac{1}{2}
$$
and by using (\ref{SI}), we obtain
\begin{align}
 ||v|^{\beta+1}|_{L^{2^{\s}_{s}}(\R^{N})}^{2} \leq  2c_{\beta}(c+|h|_{L^{\frac{N}{s}}(\R^{N})}\lambda)C^{2}_{2^{s}_{s}}  ||v|^{\beta+1}|_{L^{2}(\R^{N})}^{2}:=M_{\beta} ||v|^{\beta+1}|_{L^{2}(\R^{N})}^{2}. 
\end{align}
Now we can control the dependence on $\beta$ of $M_{\beta}$ as follows
$$
M_{\beta}\leq Cc^{2}_{\beta}\leq C(1+\beta)^{2}\leq M_{0}^{2}e^{2\sqrt{\beta+1}},
$$
which implies that
$$
|v|_{L^{2^{\s}_{s}(\beta+1)}(\R^{N})} \leq M_{0} e^{\frac{1}{\sqrt{\beta+1}}}|v|_{L^{2(\beta+1)}(\R^{N})}.
$$
Iterating this last relation and choosing $\beta_{0}=0$ and $2(\beta_{n+1}+1)=2^{\s}_{s}(\beta_{n}+1)$, we deduce that
$$
|v|_{L^{2^{\s}_{s}(\beta_{n}+1)}(\R^{N})} \leq M_{0}^{\sum_{i=0}^{n} \frac{1}{\beta_{i}+1}} e^{\sum_{i=0}^{n} \frac{1}{\sqrt{\beta_{i}+1}}}|v|_{L^{2(\beta_{0}+1)}(\R^{N})}.
$$
We note that $1+\beta_{n}=(\frac{N}{N-2s})^{n}$, so the series
$$
\sum_{i=0}^{\infty} \frac{1}{\beta_{i}+1} \mbox{ and } \sum_{i=0}^{\infty} \frac{1}{\sqrt{\beta_{i}+1}}
$$
are finite and we get
$$
|v|_{L^{\infty}(\R^{N})}=\lim_{n\rightarrow \infty} |v|_{L^{2^{\s}_{s}(\beta_{n}+1)}(\R^{N})}<\infty.
$$
\end{proof}

\begin{thm}\label{reg}
Let $u\in H^{s}(\R^{N})$ be a solution to (\ref{P}). Then $u\in C^{1,\beta}(\R^{N})$ for some $\beta\in (0,1)$ and $u(x)\rightarrow 0$ as $|x| \rightarrow \infty$.
\end{thm}
\begin{proof}
We proceed as in the proof of Lemma 4.4 in \cite{CS1}. Let $g(x)=[m^{2s}u-\mu u+|u|^{p-2}u](x,0)$. By Lemma \ref{lemmino} we know that $g \in L^{q}(\R^{N})$ for any $q\in [2,\infty]$. By Bessel potential theory (see \cite{Stein}) we have $u\in L^{2s}_{q}(\R^{N})$ for any $q<\infty$ so $u\in C^{0,\alpha}(\R^{N})$ for some $\alpha \in (0,1)$ and $u(x)\rightarrow 0$ as $|x| \rightarrow \infty$. Then $g\in C^{0,\alpha}(\R^{N})$. Hence, if $2s+\alpha>1$ then $u\in C^{1,2s+\alpha-1}(\R^{N})$ and if $2s+\alpha\leq 1$ then $u\in C^{0,2s+\alpha}(\R^{N})$. Therefore, iterating the procedure a finite number of times, one gets that $u\in C^{1,\sigma}(\R^{N})$ for some $\sigma \in (0,1)$ depending only on $s$. 
\end{proof}

\begin{thm}
Every ground state $u$ of (\ref{P}) has one sign.
\end{thm}
\begin{proof}
Let $v$ be a unique solution to (\ref{extpb}) with data $u$. Then $v\in \mathcal{N}_{m}$ and $\I_{m}(v)=c_{m}$. In particular $|v|\in \mathcal{N}_{m}$ and $\I_{m}(|v|)=\I_{m}(v)$, that is $|v|$ is a weak solution to (\ref{fracvaldinoci}). Then $|v|$ and $y^{1-2s}\partial_{y}|v|$ are continuous up to the boundary.  Let assume by contradiction that $|v|$ achieves its global minimum at $(x_{0},0)\in \partial \R^{N+1}_{+}$. By using Hopf principle we can deduce that $-y^{1-2s}\frac{\partial |v|}{\partial y}(x_{0},0)<0$. This gives a contradiction since 
$$
-y^{1-2s}\frac{\partial |v|}{\partial y}(x_{0},0)=[m^{2s}-\mu] |v|(x_{0},0)+|v|^{p-1}(x_{0},0)=0.
$$
Then $u>0$ or $u<0$ in $\R^{N}$.
\end{proof}

\begin{thm}
Every positive solution $u\in H^{s}(\R^{N})$ of (\ref{P}) is radially symmetric with respect to some point $x_{0}\in \R^{N}$.
\end{thm}
\begin{proof}
We proceed as in \cite{ChS}. Let $v$ be a unique solution of (\ref{extpb}) with boundary data $u$. 

Let $\lambda>0$ and we consider the sets
$$
R_{\lambda}=\{ (x_{1}, \dots, x_{N},y):x_{1}>\lambda , y\geq 0 \}
$$
and
$$
T_{\lambda}=\{(x_{1},\dots, x_{N}, y):x_{1}=\lambda , y\geq 0 \}.
$$
Let $v_{\lambda}(x,y)=v(2\lambda-x_{1},\dots,x_{N},y)$ and $w_{\lambda}=v_{\lambda}-v$.
Then $w_{\lambda}$ satisfies

\begin{equation}\label{eqwl}
\left\{
\begin{array}{ll}
-\dive(y^{1-2s}\nabla w_{\lambda})+m^{2}y^{1-2s} w_{\lambda}=0  &\mbox{ in } \, \R^{N+1}_{+}  \\
\frac{\partial w_{\lambda}}{\partial \nu^{a}}= (C_{\lambda}(x) +m^{2s} -\mu) w_{\lambda}(x,0)  &\mbox{ on } \, \partial \R^{N+1}_{+}
\end{array},
\right.
\end{equation}
where 
\begin{equation}
C_{\lambda}(x) :=
\left\{ 
\begin{array}{cc}
\frac{v_{\lambda}^{p-1}(x,0) -v^{p-1}(x,0)}{v_{\lambda}(x,0)-v(x,0)} &\mbox{ if } v_{\lambda}(x,0)\neq v(x,0) \\
0 &\mbox{ if } v_{\lambda}(x,0)= v(x,0)
\end{array}.
\right.
\end{equation}

Let $w_{\lambda}^{-}:=\min\{0, w_{\lambda}\}$. Note that as $\lambda\rightarrow \infty$, $C_{\lambda}(x) \rightarrow 0$ uniformly for $x$ such that $w_{\lambda}^{-}\neq 0$ because of $\lim_{|x| \rightarrow \infty} |v(x,0)|=0$ and $0\leq v_{\lambda}<v$.
Multiplying the weak formulation of (\ref{eqwl}) by $w_{\lambda}^{-}$ and applying inequality (\ref{SI}), we get for for $\lambda>0$ sufficiently large
\begin{align*}
&\iint_{R_{\lambda}} y^{1-2s} [|\nabla w^{-}_{\lambda} |^{2}+m^{2}(w^{-}_{\lambda})^{2}] dx dy=\int_{\{ x_{1}>\lambda\}} [C_{\lambda}(x) +m^{2s} -\mu] (w^{-}_{\lambda})^{2} dx  \\
&\leq \int_{\{ x_{1}>\lambda\}} C_{\lambda}(x) (w^{-}_{\lambda})^{2} dx+ A(m,\mu,s) \iint_{R_{\lambda}} y^{1-2s} [|\nabla w^{-}_{\lambda} |^{2}+m^{2}(w^{-}_{\lambda})^{2}] dx dy
\end{align*} 
where $A(m,\mu,s)=(1-\frac{\mu}{m^{2s}})$ if $m^{2s}-\mu>0$ and $A(m,\mu,s)=0$ otherwise.
Then $w_{\lambda}^{-}\equiv 0$ on $R_{\lambda}$ and as a consequence 
$$
w_{\lambda}(x,y)\geq 0 \, \mbox{ on } R_{\lambda} 
$$
for $\lambda>0$ sufficiently large.

Now, we define
$$
\nu :=\inf\{ \tau>0 : w_{\lambda}\geq 0 \mbox{ on } R_{\lambda} \mbox{ for every } \lambda \geq \tau \}. 
$$
We distinguish two cases.
We begin assuming $\nu>0$. We want to prove that $w_{\nu}\equiv 0$. We argue by contradiction. 

By continuity, $w_{\nu}\geq 0$ on $R_{\nu}$, and by the strong maximum principle
$w_{\nu}>0$ on the set 
$$
R_{\nu}^{'} :=\{(x_{1}, \dots, x_{N}, y) : x_{1} >\nu, \, x_{i}\in \R \, (i=1, \dots, N)\, y>0\}.
$$

We also have $w_{\nu}(x,0)\geq 0$ on the set $\{x\in \R^{N} : x_{1} \geq \nu\}$ by continuity. 
Furthermore, by Hopf principle $w_{\nu}(x,0)>0$ on the set $\{x\in \R^{N} : x_{1}>\nu\}$. 

Take $\lambda_{j}<\nu$ such that $\lambda_{j}\rightarrow \nu$ as $j\rightarrow \infty$. 
Let $r_{0}>0$ such that $|C_{\nu}(x)|\leq \frac{\mu}{4}$ for every $|x|>r_{0}$. Since $||w_{\lambda_{j}}||_{C^{1}(\R^{N+1}_{+})}$ is uniformly bounded, $D:= |C_{\lambda_{j}}|_{L^{\infty}(\R^{N})}<\infty$ and $|C_{\lambda_{j}}(x)|\leq \frac{\mu}{2}$ for every $|x|>r_{0}$ and $j\in \N$. 

We denote by 
$$
B_{r_{0}}(p_{j})=\{ x\in \R^{N}: |x-p_{j}|<r_{0} \}\subset \R^{N}
$$ 
where $p_{j}=(\lambda_{j}, 0, \dots, 0)$. 
As above, we obtain
\begin{align*}
&\iint_{R_{\lambda_{j}}} y^{1-2s}[|\nabla w_{\lambda_{j}}^{-}|^{2} +m^{2}(w_{\lambda_{j}}^{-})^{2}] dxdy \leq \int_{\{x_{1} >\lambda_{j}\}} (C_{\lambda_{j}} +m^{2s}-\mu)(w_{\lambda_{j}}^{-})^{2} dx \\ \nonumber
&\leq (D+m^{2s}+\mu)\int_{\{x_{1}>\lambda_{j}\}\cap B_{r_{0}}(p_{j})} (w_{\lambda_{j}}^{-})^{2} dx + (m^{2s}-\frac{\mu}{2}) \int_{\{x_{1}>\lambda_{j}\} \setminus B_{r_{0}}(p_{j})} (w_{\lambda_{j}}^{-})^{2} dx \\ \nonumber
&\leq (D+m^{2s}+\mu)\int_{\{x_{1}>\lambda_{j}\}\cap B_{r_{0}}(p_{j})} (w_{\lambda_{j}}^{-})^{2} dx + B(m,\mu,s) \int_{\{x_{1}>\lambda_{j}\} \setminus B_{r_{0}}(p_{j})} (w_{\lambda_{j}}^{-})^{2} dx .  
\end{align*}
where $B(m,\mu,s)=(m^{2s}-\frac{\mu}{2})$ if $m^{2s}-\frac{\mu}{2}>0$ and it is zero otherwise.
Since $w_{\nu}(x)>0$ on the set $\{x\in \R^{N} : x_{1}>\nu\}$, the measure of set $E_{j} =\{x\in B_{r_{0}}(p_{j}): w_{\lambda_{j}}^{-}(x,0)\neq 0 \}$ goes to $0$ as $j\rightarrow \infty$. 

Then using H\"older and Sobolev inequality, we see 
\begin{align*}
\int_{\{x_{1}>\lambda_{j}\}\cap B_{r_{0}}(p_{j})} (w_{\lambda_{j}}^{-})^{2} dx
&= \int_{\{x_{1}>\lambda_{j}\}} \chi_{E_{j}} (w_{\lambda_{j}}^{-})^{2} dx  \\
&\leq |\chi_{E_{j}}|_{L^{N/2s}}  |w_{\lambda_{j}}^{-}|^{2}_{L^{2N/N-2s}} \\
&\leq  o(1)\int_{R_{\lambda_{j}}} |\nabla w_{\lambda_{j}}^{-}|^{2} dxdy. 
\end{align*}

Therefore $w_{\lambda_{j}}\geq 0$ on $R_{\lambda_{j}}$, if $j$ is large. This gives a contradiction because of the minimality of $\nu$. Thus we can conclude that $w_{\nu}\equiv 0$ on $R_{\nu}$ and we get the symmetry with respect to the $x_{1}$ direction. 

Now assume $\nu=0$. By repeating the above argument for $\lambda<0$ and $w_{\lambda}:= v_{\lambda}-v$ defined on 
$$
L_{\lambda} :=\{ (x_{1}, \dots, x_{N}, y) \in \R^{N+1}_{+} : x_{1}<\lambda, \, x_{i}\in \R \, (i=1, \dots, n), \, y\geq 0	\}. 
$$ 
Then $w_{\lambda}\geq 0$ for $|\lambda|$ sufficiently large. 
Let  
$$
\nu' := \sup\{\tau <0 : W_{\lambda} \geq 0 \mbox{ on } L_{\lambda} \mbox{ for every } \lambda\leq \tau \}. 
$$
If $\nu'<0$, we get the symmetry as above. If $\nu'=0$, by using $\nu=0$ we have
$$
v(-x_{1}, x_{2}, \dots, x_{N}, y)\geq v(x_{1}, x_{2}, \dots, x_{N}, y) \mbox{ on } \R^{N+1}_{+}. 
$$
Consequently, by replacing $x_{1}$ with $-x_{1}$ we deduce that  
$$
v(-x_{1}, x_{2}, \dots, x_{N}, y)= v(x_{1}, x_{2}, \dots, x_{N}, y) \mbox{ on } \R^{N+1}_{+}. 
$$

Using the same approach in any arbitrary direction $x_{i}$, we conclude the proof. 

\end{proof}

\section{passage to the limit as $m \rightarrow 0$}

In this section we show that it is possible to pass to the limit in problem (\ref{fracvaldinoci}) and to find a nontrivial ground state to (\ref{DPV}).
In order to prove this, we estimate $c_{m}$ defined in Section $1$ from above and below 
uniformly in $m$.

Fix $0<m<(\frac{\mu}{2})^{1/2s}$.
By using the characterization of the infimum $c_{m}$ on $\mathcal{N}_{m}$ we can see that
$$
c_{m}=\inf_{v\in \mathcal{N}_{m}} \I_{m}(v)=\inf_{v\in X_{m,rad}\setminus \{0\}} \max_{t>0} \I_{m}(tv).
$$
Now, we prove that there exists $\lambda>0$ and $\delta>0$ independent from $m$ such that
\begin{equation}\label{ub}
\lambda \leq c_{m} \leq \delta.
\end{equation}
Let 
$$
w(x,y)=v_{0}(x) \frac{1}{y+1}
$$
where $v_{0}$ is defined by setting
\begin{equation}\label{vdef}
v_{0}(x)=
\left\{
\begin{array}{ll}
1 &\mbox{ if } |x|\leq 1 \\
2-|x| &\mbox{ if } 1\leq |x|\leq 2 \\ 
0 &\mbox{ if } |x|\geq 2 
\end{array}.
\right.
\end{equation}

Then $w\in H^{1}_{m}(\R^{N+1}_{+},y^{1-2s})$ and 
\begin{align}
||w||^{2}_{H^{1}_{m}(\R^{N+1}_{+},y^{1-2s})}&=\Bigl(\int_{0}^{\infty} y^{1-2s} \frac{1}{(y+1)^{2}} dy \Bigr) \Bigl[\int_{\R^{N}} |\nabla_{x}v_{0}|^{2} +m^{2}v^{2}_{0} dx \Bigr] \nonumber \\
&+\Bigl(\int_{0}^{\infty} y^{1-2s} \frac{1}{(y+1)^{4}} dy \Bigr) \int_{\R^{N}} v^{2}_{0} dx \nonumber \\
&\leq A \Bigl[\int_{\R^{N}} |\nabla_{x}v_{0}|^{2} +\Bigl(\frac{\mu}{2}\Bigr)^{1/s}v^{2}_{0} dx \Bigr]+B \int_{\R^{N}} v^{2}_{0} dx=:C.
\end{align}

Hence we have
\begin{align*}
\sup_{t>0} \I_{m}(tw) &=\Bigl(\frac{1}{2}-\frac{1}{p} \Bigr) \frac{||w||_{e,m}^{\frac{p}{p-2}}}{|w|_{p}^{\frac{2}{p-2}}} \\
&\leq \Bigl(\frac{1}{2}-\frac{1}{p} \Bigr) \frac{[||w||_{H^{1}_{m}(\R^{N+1}_{+},y^{1-2s})}^{2}+\mu|v_{0}|^{2}_{L^{2}(\R^{N})}]^{\frac{p}{p-2}}}{|v_{0}|_{L^{p}(\R^{N})}^{\frac{2}{p-2}}} \\
& \leq \Bigl(\frac{1}{2}-\frac{1}{p} \Bigr) \frac{[C+\mu|v_{0}|_{L^{2}(\R^{N})}^{2}]^{\frac{p}{p-2}}}{|v_{0}|_{L^{p}(\R^{N})}^{\frac{2}{p-2}}}=:\delta 
\end{align*}
Since $\I_{m}(v_{m})=c_{m}$ and $\J_{m}(v_{m})=0$ we know that
$$
c_{m}=\I_{m}(v_{m})=\Bigl(\frac{1}{2}-\frac{1}{p} \Bigr) \int_{\R^{N}} |v_{m}|^{p} dx.
$$
Then to deduce a lower bound to $c_{m}$ it is enough to estimate the $L^{p}(\R^{N})$ norm of $v_{m}$.
Since $\J_{m}(v_{m})=0$ we can see that
\begin{align*}
|v_{m}|_{L^{p}(\R^{N})}^{p}&=\iint_{\R^{N+1}_{+}} y^{1-2s} (|\nabla v_{m}|^{2}+m^{2} v_{m}^{2}) dx dy+(\mu-m^{2s})  \int_{\R^{N}} v_{m}^{2} dx \\
&\geq \iint_{\R^{N+1}_{+}} y^{1-2s} |\nabla v_{m}|^{2} dx dy+\frac{\mu}{2}|v_{m}|_{L^{2}(\R^{N})}^{2} \\
&\geq \kappa_{s}[v_{m}]^{2}_{H^{s}(\R^{N})}+\frac{\mu}{2}|v_{m}|_{L^{2}(\R^{N})}^{2} \\
&\geq C_{s,\mu} |v_{m}|^{2}_{H^{s}(\R^{N})}\geq C'_{s,\mu}|v_{m}|_{L^{p}(\R^{N})}^{2}
\end{align*}
that is 
\begin{align}\label{nozero} 
|v_{m}|_{L^{p}(\R^{N})} \geq (C'_{s,\mu})^{\frac{1}{p-2}} \mbox{ and }  c_{m}\geq \Bigl(\frac{1}{2}-\frac{1}{p} \Bigr)(C'_{s,\mu})^{\frac{p}{p-2}}=:\lambda. 
\end{align}

Now, by using (\ref{ub}), we are able to prove the following result

\begin{thm}
There exists $v\in H^{1}_{loc}(\R^{N+1}_{+},y^{1-2s})$ such that
$v_{m} \rightharpoonup v$ in $L^{2}(\R^{N} \times [0,\varepsilon],y^{1-2s})$ for any $\varepsilon>0$, $\nabla v_{m} \rightharpoonup \nabla v$ in $L^{2}(\R^{N+1}_{+},y^{1-2s})$ and $v_{m}(\cdot,0) \rightarrow v(\cdot,0)$ in $L^{q}(\R^{N})$ for any $q\in (2, \frac{2N}{N-2s})$, as $m\rightarrow 0$. In particular $v(\cdot,0)$ is a nontrivial weak solution to (\ref{DPV}). 
\end{thm}
\begin{proof}
Taking into account $\J(v_{m})=0$, $c_{m}\leq \delta$ and $0<m<(\frac{\mu}{2})^{1/2s}$ we can see that 
\begin{align}
\delta^{1/p}\Bigl(\frac{1}{2}-\frac{1}{p}\Bigr)^{-1/p}& \geq 
c_{m}^{1/p}\Bigl(\frac{1}{2}-\frac{1}{p}\Bigr)^{-1/p}=|v_{m}|_{L^{p}(\R^{N})}^{p} \nonumber \\
&=\iint_{\R^{N+1}_{+}} y^{1-2s} (|\nabla v_{m}|^{2}+m^{2} v_{m}^{2}) dx dy+(\mu-m^{2s})  \int_{\R^{N}} v^{2} dx  \nonumber \\
&\geq \iint_{\R^{N+1}_{+}} y^{1-2s} |\nabla v_{m}|^{2} dx dy+\frac{\mu}{2}|v_{m}|_{L^{2}(\R^{N})}^{2} \nonumber \\
&\geq \kappa_{s}[v_{m}]^{2}_{H^{s}(\R^{N})}+\frac{\mu}{2}|v_{m}|_{L^{2}(\R^{N})}^{2} \nonumber \\
&\geq C_{s,\mu} |v_{m}|^{2}_{H^{s}(\R^{N})}.
\end{align}
that is 
$$
\iint_{\R^{N+1}_{+}} y^{1-2s} |\nabla v_{m}|^{2} dx dy\leq C_{1}
$$
and
$$
 |v_{m}|^{2}_{H^{s}(\R^{N})} \leq C_{2}.
$$
Now, fix $\varepsilon>0$ and let $v\in \mathcal{C}^{\infty}_{c}(\overline{\R^{N+1}_{+}})$ such that $||v||_{H^{1}_{m}(\R^{N+1}_{+},y^{1-2s})}<\infty$.
For any $x\in \R^{N}$ and $y\in [0, \varepsilon]$, we have
$$
v(x,y)=v(x,0)+\int_{0}^{y} \partial_{y} v(x,t) dt.
$$
By using $(a+b)^{2}\leq 2a^{2}+2b^{2}$ for all $a, b\geq 0$ we obtain
$$
|v(x,y)|^2 \leq 2  |v(x,0)|^{2}+2\Bigl(\int_{0}^{y}|\partial_{y} v(x,t)| dt\Bigr)^{2},
$$
and by applying the H\"older inequality we deduce
\begin{equation*}
|v(x,y)|^2 \leq 2 \Bigl[ |v(x,0)|^{2}+\Bigl(\int_{0}^{y} t^{1-2s}|\partial_{y} v(x,t)|^{2}dt\Bigr)\frac{y^{2s}}{2s}\,  \Bigr].
\end{equation*}
Multiplying both members by $y^{1-2s}$ we have
\begin{equation}\label{vtii}
y^{1-2s}|v(x,y)|^2 \leq 2 \Bigl[ y^{1-2s}|v(x,0)|^{2}+\Bigl(\int_{0}^{y} t^{1-2s} |\partial_{y} v(x,t)|^{2}dt\Bigr)\frac{y}{2s} \Bigr].
\end{equation}
Integrating (\ref{vtii}) over $\R^{N}\times [0,\varepsilon]$ we have
\begin{align}\label{nash}
||v||_{L^{2}(\R^{N}\times [0,\varepsilon],y^{1-2s})}^2 &\leq \frac{\varepsilon^{2-2s}}{1-s} |v(\cdot,0)|_{L^{2}(\R^{N})}^{2}+\frac{\varepsilon^{2}}{2s} ||\partial_{y} v||_{L^{2}(\R^{N+1}_{+},y^{1-2s})}^{2}.
\end{align}
By density (\ref{nash}) holds for any $v\in H^{1}_{m}(\R^{N+1}_{+},y^{1-2s})$.

Then, replacing $v$ by $v_{m}$ we can infer that 
\begin{align*}
||v_{m}||_{L^{2}(\R^{N}\times [0,\varepsilon],y^{1-2s})}^2 \leq C(\varepsilon,s) K(\delta,p)^{2}
\end{align*}
for any $0<m<(\frac{\mu}{2})^{1/2s}$. Then there exists $v\in H^{1}_{loc}(\R^{N+1}_{+},y^{1-2s})$ such that 
\begin{align}
&v_{m} \rightharpoonup v \mbox{ in } L^{2}(\R^{N} \times [0,\varepsilon],y^{1-2s}) \label{5.8} \mbox{ for all } \varepsilon>0 \\
&\nabla v_{m} \rightharpoonup \nabla v \mbox{ in } L^{2}(\R^{N+1}_{+},y^{1-2s}) \label{5.9}\\
&v_{m} (\cdot, 0) \rightarrow v(\cdot, 0) \mbox{ in } L^{q}(\R^{N}) \quad \forall q\in \Bigl(2, \frac{2N}{N-2s} \Bigr) \label{5.10}  
\end{align}
as $m \rightarrow 0$.
Finally, we prove that $v(\cdot,0)$ is a nontrivial weak solution to (\ref{DPV}). 
We proceed as in \cite{A2}. 
Fix $\eta\in \mathcal{C}^{\infty}_{c}(\overline{\R^{N+1}_{+}})$ such that $\nabla \eta\in L^{2}(\R^{N+1}_{+},y^{1-2s})$ and let $\psi\in C^{\infty}([0,\infty))$ defined by 
\begin{equation}\label{xidef}
\left\{
\begin{array}{ll}
\psi=1 &\mbox{ if } 0\leq y\leq 1 \\
0\leq \psi \leq 1 &\mbox{ if } 1\leq y\leq 2 \\ 
\psi=0 &\mbox{ if } y\geq 2 .
\end{array}
\right.
\end{equation}
Let $\psi_{R}(y)=\psi(\frac{y}{R})$ for $R>1$; then $\eta \psi_{R} \in H^{1}_{m}(\R^{N+1}_{+}, y^{1-2s})$. 

Then putting $\eta \psi_{R}$ in the weak formulation of (\ref{fracvaldinoci}) we have
\begin{align*}
&\iint_{\R^{N+1}_{+}} y^{1-2s} [\nabla v_{m} \nabla(\eta \psi_{R})+m^{2} v_{m} \eta\psi_{R}] dx dy\\
&+(\mu-m^{2s})  \int_{\R^{N}} v_{m}\eta dx=\int_{\R^{N}} |v_{m}|^{p-1}\eta dx.
\end{align*}
Taking the limit as $m\rightarrow 0$ and by using  (\ref{5.8})-(\ref{5.10}) we find
\begin{align*}
\iint_{\R^{N+1}_{+}} y^{1-2s} \nabla v \nabla(\eta \psi_{R}) dx dy+\mu  \int_{\R^{N}} v\eta dx=\int_{\R^{N}} |v|^{p-1}\eta dx.
\end{align*}
By passing to the limit as $R \rightarrow \infty$ we deduce that
\begin{align*}
\iint_{\R^{N+1}_{+}} y^{1-2s} \nabla v \nabla\eta  dx dy+\mu  \int_{\R^{N}} v\eta dx=\int_{\R^{N}} |v|^{p-1}\eta dx
\end{align*}
for any $\eta\in \mathcal{C}^{\infty}_{c}(\overline{\R^{N+1}_{+}})$ such that $\nabla \eta\in L^{2}(\R^{N+1}_{+},y^{1-2s})$.
Finally $v(\cdot,0)$ is not identically zero because of  (\ref{nozero}), (\ref{5.10}) and $2<p<\frac{2N}{N-2s}$.
Proceeding as in \cite{FQT} we can show that $v$ is a positive radially symmetric function.
\end{proof}

\end{document}